\documentclass{amsart}
\usepackage{amsfonts,amssymb,amscd,amsmath,enumerate,verbatim,calc}
\usepackage[all]{xy}

\newcommand{\CM}{Cohen-Macaulay}

\newcommand{\B}{\mathcal{B} }

\newcommand{\fb}{\mathbf{f} }
\newcommand{\n}{\mathfrak{n} }
\newcommand{\m}{\mathfrak{m} }

\newcommand{\xb}{\mathbf{x} }

\newcommand{\Z}{\mathbb{Z} }
\newcommand{\Cb}{\mathbb{C} }
\newcommand{\C}{\mathcal{C} }
\newcommand{\D}{\mathcal{D} }
\newcommand{\T}{\mathcal{T} }
\newcommand{\I}{\mathcal{I} }
\newcommand{\Q}{\mathbb{Q} }
\newcommand{\rt}{\rightarrow}

\newcommand{\ov}{\overline}

\newcommand{\wh}{\widehat }

\newcommand{\image}{\operatorname{image}}
\newcommand{\mmod}{\operatorname{mod}}
\newcommand{\projdim}{\operatorname{projdim}}
\newcommand{\cone}{\operatorname{cone}}
\newcommand{\height}{\operatorname{height}}

\newcommand{\curv}{\operatorname{curv}}
\newcommand{\cx}{\operatorname{cx}}

\newcommand{\rank}{\operatorname{rank}}
\newcommand{\coker}{\operatorname{coker}}

\newcommand{\Spec}{\operatorname{Spec}}
\newcommand{\CMS}{\operatorname{\underline{CM}}}
\newcommand{\CMa}{\operatorname{CM}}
\newcommand{\Ass}{\operatorname{Ass}}

\newcommand{\Hom}{\operatorname{Hom}}

\newcommand{\Syz}{\operatorname{Syz}}
\newcommand{\sHom}{\operatorname{\underline{Hom}}}

\theoremstyle{plain}

\newtheorem{theorem}{Theorem}[section]
\newtheorem{corollary}[theorem]{Corollary}
\newtheorem{lemma}[theorem]{Lemma}
\newtheorem{proposition}[theorem]{Proposition}

\theoremstyle{definition}
\newtheorem{definition}[theorem]{Definition}
\newtheorem{s}[theorem]{}
\newtheorem{remark}[theorem]{Remark}

\theoremstyle{remark}

\begin{document}

\title[Grothendieck Groups]{ Grothendieck groups and completions of Gorenstein local rings}
\author{Tony~J.~Puthenpurakal}
\date{\today}
\address{Department of Mathematics, IIT Bombay, Powai, Mumbai 400 076, India}

\email{tputhen@gmail.com}

 \begin{abstract}
Let $(A,\m)$ be an excellent Gorenstein local ring of dimension $d \geq 2$ which is an isolated singularity. Let $\wh{A}$ denote the completion of $A$. If $G(A)$ is the Grothendieck group of $A$ then by $G(A)_\Q$ we denote $G(A)\otimes_\Z \Q$.
We prove that the natural map $G(A)_\Q \rt G(\wh{A})_\Q$ is an isomorphism if and only if for any maximal \CM (= MCM) $\wh{A}$-module  $M$ there exists an MCM $A$-module $N$ and integers $r \geq 1$ and $s \geq 0$ (depending on $M$) such that $M^r\oplus \wh{A}^s \cong \wh{N}$. An essential ingredient is the classification of $\Q$-subspaces of $G(\C)_\Q$ (here $\C$ is a skelletaly small triangulated category)
in terms of certain dense subcategories of  $\C$.  We also give criterion for a Henselian Gorenstein ring $B$ (not an isolated singularity) such that the natural map $G(B)_\Q \rt G(\wh{B})_\Q$ is an isomorphism ( when $\dim B = 2, 3$). We give many examples where our result holds.
\end{abstract}
 \maketitle
\section{introduction}
If $H$ is an abelian group let $H_\Q = H\otimes_\Z \Q$.  Furthermore if $f \colon H_1 \rt H_2$ is a homomorphism of abelian groups then let $f_\Q$ denote
$f\otimes 1_\Q$.
Let $(A,\m)$ be a Noetherian local ring. Let $G(A)$ be the Grothendieck group of $A$. Let $\wh{A}$ denote the completion of $A$. We have a natural map $\eta \colon G(A) \rt G(\wh{A})$. In general $\eta$ need not be injective, see \cite{D}.  There are also examples when $\eta_\Q$ is not injective, see \cite{KS}. However if $A$ is a homomorphic image of an excellent regular local ring
and an isolated singularity then $\eta$ is injective, see \cite[1.5(iii)]{KK}.

In this paper we first investigate the case when $\eta_\Q$ is an isomorphism when $(A, \m)$ is an excellent Gorenstein isolated singularity of dimension $d \geq 2$. By an MCM $A$-module $M$ we mean a maximal \CM \ $A$-module (i.e., $M$ is \CM \ and $\dim M = \dim A$). We prove
\begin{theorem}\label{main}
Let $(A,\m)$ be an excellent Gorenstein isolated singularity of dimension $d \geq 2$. Let $\eta \colon G(A) \rt G(\wh{A})$ be the natural map. The following assertions are equivalent:
\begin{enumerate}[\rm (i)]
  \item $\eta_\Q$ is an isomorphism.
  \item For any MCM $\wh{A}$-module  $M$ there exists an MCM $A$-module $N$ and integers $r \geq 1$ and $s \geq 0$ (depending on $M$) such that $M^r\oplus \wh{A}^s \cong \wh{N}$.
\end{enumerate}
\end{theorem}
\begin{remark}
\begin{enumerate}
  \item We show that in general $\eta$ is injective (here $A$ need not be an image of an  excellent regular ring). So $\eta_\Q$ is injective.
  \item The assertion  (ii) $\implies$ (i) follows easily from (1). The content of the theorem is (i) $\implies$ (ii).
\end{enumerate}
\end{remark}

\s\label{stable} Let $(A,\m)$ be an excellent Gorenstein isolated  singularity of dimension $d \geq 2$. Let $\CMS(A)$ denote the stable category of maximal \CM \ $A$-modules. We note that $\CMS(A)$ is a triangulated category, see \cite[4.4.7]{B}. Let $G(\CMS(A))$ be its Grothendieck group. We have a natural map $\theta \colon G_0(\CMS(A)) \rt G_0(\CMS(\wh{A}))$. Under the assumptions of Theorem \ref{main} it is not difficult to prove that $\theta$ is injective. Furthermore we show  that $\eta $ (resp. $\eta_\Q$) is an isomorphism if and only if $\theta$ (resp $\theta_\Q$) is an isomorphism. We note that the assumption $\theta$ is an isomorphism implies that
the natural map $\CMS(A) \rt \CMS(\wh{A})$ is an equivalence, see \cite[2.4]{P}. Theorem \ref{main} follows from the following:
\begin{theorem}\label{main-st}
(with hypotheses as in \ref{stable}). The following assertions are equivalent:
\begin{enumerate}[\rm (i)]
  \item $\theta_\Q$ is an isomorphism.
  \item For any MCM $\wh{A}$-module  $M$ there exists an MCM $A$-module $N$ and an integer $r \geq 1$ (depending on $M$) such that $M^r\cong \wh{N}$ in $\CMS(\wh{A})$.
\end{enumerate}
\end{theorem}

\s Theorem \ref{main-st} follows from a more general result on 1-1 correspondence of $\Q$-subspaces of the Grothendieck group $G(\C)_\Q$ of a triangulated category with certain dense subcategories of $\C$.
We recall a result due to Thomason \cite{T}. Let $\C$ be a skeletally small triangulated category. Recall a subcategory $\D$ is dense in $\C$ if the smallest thick subcategory of $\C$ containing $\D$ is $\C$ itself. In \cite{T} a one-to one correspondence between dense subcategories of $\C$ and subgroups of the Grothendieck group $G_0(\C)$ is given.
Our interest was in finding a  similar correspondence for $\Q$-subspaces of $G(\C)_\Q$.
\begin{definition}\label{def-rad}
A dense subcategory $\D$ of $\C$ is said to be a radical dense sub-category of $\C$ if $U^n \in \D$ for some $n \geq 1$ implies $U \in \D$.
\end{definition}
It is easy to construct radical dense subcategories of $\C$. We show
\begin{proposition}\label{sqrt}
Let  $\D$ be a dense subcategory of $\C$. Let
$$\sqrt{\D} = \{ U \in \C \mid U^n \in \D \ \text{for some $n \geq 1$} \}. $$
Then
\begin{enumerate}[\rm (1)]
\item $\D \subseteq \sqrt{\D}$.
  \item $\sqrt{\sqrt{D}} = \sqrt{\D}$.
  \item $\sqrt{\D}$ is a radical dense subcategory of $\C$.
   \item $\D$ is a radical dense subcategory of $\C$ if and only if $\sqrt{\D} = \D$.
\end{enumerate}
\end{proposition}
Recall if $\D$ is a dense subcategory of $\C$ then the natural map $\delta^\D \colon G(\D) \rt G(\C)$ is injective. So we have an inclusion $\delta^\D_\Q \colon G(\D)_\Q \rt G(\C)_\Q$.
We show that $\image \delta^\D_\Q = \image \delta^{\sqrt{\D}}_\Q$, see \ref{equal}. We prove
\begin{theorem}\label{main-cat}
Let $\C$ be a skeletally small triangulated category.
\begin{enumerate}[\rm (1)]
  \item Let $H$ be a $\Q$-subspace of $G(\C)_\Q$. Then there exists a radical  dense subcategory $\D$ of $\C$ such that $\image \delta^\D_\Q = H$.
  \item If $\D_1, \D_2$ are two radical dense subcategories of $\C$ such that $H = \image \delta^{\D_1}_\Q = \image \delta^{\D_2}_\Q$ then $\D_1 = \D_2$. Thus we may write $\D_H$ to be the unique radical  dense subcategory of $\C$ with $\image \delta^\D_\Q = H$.
\item Let $H_1, H_2$ be $\Q$-subspace's of $G(\C)_\Q$. We have $\D_{H_1} \subseteq \D_{H_2}$ if and only if $H_1 \subseteq H_2$.
\end{enumerate}
\end{theorem}

\s \label{sub} There are many naturally defined thick subcategories of $\C = \CMS(A)$. We consider pairs $\D, \wh{\D}$ of thick subcategories of $\C$ and $\wh{C} = \CMS(\wh{A})$ respectively with the following properties:
\begin{enumerate}
  \item If $M \in \D$ then $\wh{M} \in \wh{\D}$
  \item If $\wh{M} \in \wh{\D}$ then $M \in \D$.
\end{enumerate}
Let $\T = \C/\D$ and $\wh{\T} = \wh{\C}/\wh{\D}$. the Verdier quotient. We have natural maps $\theta \colon G(\C) \rt G(\wh{\C})$, $\alpha \colon G(\D) \rt G(\wh{\D})$ and $\beta \colon G(\T) \rt G(\wh{\T})$.

For example
\begin{enumerate}
  \item $A$ is a complete intersection of codimension $c$. Let For $1\leq i \leq c$ let $\D_i = \CMS^{\leq i}(A)$ consisting of MCM modules of complexity $\leq i$. Set $\wh{\D_i} = \CMS^{\leq i}(\wh{A})$.
  \item Let $A$ be a Gorenstein ring which is not a complete intersection. Then note that $\curv_A k  > 1$, see \cite[8.2.1]{A}. Also $\curv_{\wh{A}} k = \curv_A k$. For $1 < \beta \leq  \curv_A k$ let $\D_\beta$ (resp. $\wh{\D}_\beta$) be the set consisting of MCM $A$-modules (resp. MCM $\wh{A}$-modules) with curvature $< \beta$.
  \item Suppose $A = B/(\fb)$ where $(B, \n)$ is a Gorenstein local ring and $\fb = f_1, \ldots, f_r \in \n^2$ a $B$-regular sequence. We note that $\wh{A} = \wh{B}/(\fb)\wh{B}$.
  Let $\D$ (resp. $\wh{D}$) consist of MCM $A$-modules $M$ (resp. MCM $\wh{A}$-modules) with $\projdim_B M$ finite (resp. $\projdim_{\wh{B}}M$ finite).
\end{enumerate}
For definition of complexity and curvature see \ref{cx-curv}.
We prove:
\begin{theorem}(\label{sub-ver}(with hypotheses as in \ref{sub})
Let $(A,\m)$ be an excellent Gorenstein ring of dimension $d \geq 2$ which is an isolated singularity. The following assertions are equivalent:
\begin{enumerate}[\rm (i)]
  \item $\theta_\Q$ is an isomorphism.
  \item  Both $\alpha_\Q$ and $\beta_\Q$ are isomorphisms.
\end{enumerate}
\end{theorem}

\s Now assume $A$ is a quotient of a regular ring $T$. Let $A_*(A) = \bigoplus_{i = 0}^{\dim A}A_i(A)$ denote the Chow group of $A$, see \ref{chow}. By the singular Riemann-Roch theorem we have an isomorphism $\tau_{T/A} \colon G(A)_\Q \rt A_*(A)_\Q$, see \cite[Chapters 18 and 20]{F}. We note that if $A$ is an  excellent, henselian, Gorenstein isolated singularity  of dimension $d \geq 2$ and a homomorphic image of a regular ring  then $G(A) \cong G(\wh{A})$. So we have that the natural map $A_i(A)_\Q \rt A_i(\wh{A})_\Q$ is an  isomorphism for all $i \geq 0$. We note that in general $A_0(A)_\Q = 0$. We prove
\begin{theorem}\label{chow-1}
Let $(A,\m)$ be an excellent Henselian, Gorenstein ring of dimension $d \geq 2$. Assume $A$ is a quotient of a regular local ring $T$. Assume $A$ satisfies $R_{d-2}$ and that the completion of
$A$ is a domain (automatic if $d \geq 3$). Then the natural map $A_1(A)_\Q \rt  A_1(\wh{A})_\Q$ is an isomorphism.
\end{theorem}
As  easy corollaries we obtain
\begin{corollary}
  \label{chow-2}
Let $(A,\m)$ be an excellent Henselian, Gorenstein ring of dimension $2$. Assume $A$ is a quotient of a regular local ring $T$. Assume that the completion of
$A$ is a domain. Then $G(A)_\Q \cong G(\wh{A})_\Q$.
\end{corollary}
If $B$ is a normal domain then let $C(B)$ denote its class group. If $A$ is an excellent normal local domain then we have a natural map $\xi \colon C(A) \rt  C(\wh{A})$ which is injective.
\begin{corollary}
  \label{chow-3}
Let $(A,\m)$ be an excellent Henselian, Gorenstein normal domain of dimension $3$. Assume $A$ is a quotient of a regular local ring $T$.
Then the following are equivalent
\begin{enumerate}[\rm (1)]
  \item  $\eta_\Q \colon G(A)_\Q \rt G(\wh{A})_\Q$ is an isomorphism.
  \item  $\xi_\Q \colon C(A)_\Q \rt C(\wh{A})_\Q$ is an isomorphism.
\end{enumerate}
\end{corollary}

We now describe in brief the contents of this paper. In section two we discuss some preliminaries  that we need. In section three we prove Theorem \ref{main-cat}.
In section four we give proofs of Theorem \ref{main} and Theorem \ref{main-st}.
 In section five we give a proof of Theorem \ref{sub-ver}. In section six we prove Theorem \ref{chow-1}. Finally in section seven we give some examples which illustrates our results.
\section{Preliminaries}
We use \cite{N} for notation on triangulated categories. However we will assume that if $\C$ is a triangulated category then $\Hom_\C(X, Y)$ is a set for any objects $X, Y$ of $\C$.
 In this section $\C, \D$ are skeletally small triangulated category.

 \s \label{sub-tri} Let $\C$ be a triangulated category with shift functor $[1]$. A full subcategory $\D$ of $\C$ is called a \emph{triangulated subcategory} of $\C$ if
 \begin{enumerate}
   \item  $X \in \D$ then $X[1], X[-1]\in \D$.
   \item If $X \rt Y \rt Z \rt X[1]$ is a triangle in $\C$ then if $X, Y \in \D$ then so does $Z$.
   \item If $X \cong Y$ in $\C$ and $Y \in \D$ then $X \in \D$.
 \end{enumerate}

  \s A triangulated subcategory $\D$ of $\C$ is said to be \emph{thick} if $U \oplus V \in \D$ then $U, V \in \D$.
  A triangulated subcategory $\D$ of $\C$ is \emph{dense} if for any $U \in \C$ there exists $V \in \C$  such that $U \oplus V  \in \D$.
Note a triangulated subcategory $\D$ is dense in $\C$ if and only if the smallest thick subcategory of $\C$ containing $\D$ is $\C$.

\s \label{Groth} The Grothendieck group $G(\C)$ is the quotient group of the free abelian group
on the set of isomorphism classes of objects of $\C$  by the Euler relations: $[V] =
[U] + [W]$ whenever there is an exact triangle in $\C$
$$U \rt V \rt W \rt U[1]. $$
As $[U[1]] = -[U]$ in $G(\C)$, it follows that any element of $G(\C)$ is of the form $[V]$ for some $V \in \C$.

\s \label{Verd} Let $\C$ be skeletally small triangulated category and let $\D$ be a thick subcategory of $\C$. Set $\T = \C/\D$  be the Verdier quotient. There exists an exact sequence (see \cite[VIII, 3.1]{I}),
$$ G_0(\D) \rt G_0(\C) \rt G_0(\T) \rt 0. $$

 \emph{Some generalities on maps on Grothendieck groups of triangulated categories induced by exact functors}:\\
Let $\C, \D$ be triangulated categories.
\s
A triangulated  functor $F\colon \C\rt \D$ is called an \emph{equivalence up to direct summand's}
if it is fully faithful and any object $X \in \D$  is isomorphic to a direct summand of $F(Y)$ for
some $Y \in \C$.

\s \label{wc} We say $\C$ has \textit{weak cancellation}, if for $U, V, W \in \C$ we have
$$ U \oplus V \cong  U \oplus W \implies V \cong W. $$

The following results were proved in \cite{P}.
\begin{theorem}
\label{b-1} (see \cite[2.1]{P}). Let $\phi \colon \C \rt \D$ be a triangulated functor which is an  equivalence upto direct summands. Then the natural map $G(\phi) \colon G(\C) \rt G(\D)$ is injective.
\end{theorem}
We also proved
\begin{theorem}
\label{b-onto}(see \cite[2.2]{P}).
(with hypotheses as in Theorem \ref{b-1}). Also assume $\D$ satisfies weak cancellation.
If $[V] \in \image G(\phi)$ then there exists $W \in \C$ with $\phi(W) \cong V$.
\end{theorem}

\s \label{chow}  We  define the Chow group of a ring and rational equivalence. Let
$Z_i(A)$  be the free Abelian group generated by prime ideals of dimension i.
The dimension of a prime ideal $P$ is defined to be the dimension of $A/P$
as an $A$-module. If $P$ is a prime ideal with $\dim A/P = i$, we denote the
generator in $Z_i(A)$. corresponding to $P$ by $[\Spec A/P]$. The group of cycles
of A is defined by $Z_*(A) = \bigoplus_{i \geq 0} Z_i(A)$.
Let $Q$ be a prime ideal of dimension $i + 1$ and let $x$ be an element not
in $Q$. Define
\[
div(Q, x) = \sum_{\dim A/P = i}\ell_{A_P}((A/Q)/x(A/Q)_P)[\Spec A/P]
\]
Let $Rat_i(A)$ be the free Abelian subgroup of $Z_i(A)$. generated by all cycles
of the form $div(Q, x)$ with $Q$ a prime ideal of dimension $i + 1$ and $x$ not in
$Q$. Two cycles are rationally equivalent if their difference lies in $Rat_i(A)$..
This equivalence relation is called rational equivalence.
Denote $Z_i(A)/Rat_i(A)  = A_i(A)$ the $i^{th}$ Chow group of $A$. The Chow group of $A$ is the direct sum of all groups
$A_i(A)$. and is denoted by $A_*(A)$.
Let $[\Spec A/P]$ also denote the generator of $A_i(A)$ corresponding to the prime ideal $P$.
\section{Proof of Theorem \ref{main-cat}}
In this section we prove Theorem \ref{main-cat}.

\s\label{T-const} (with setup as in \ref{Groth}). Thomason \cite[2.1]{T} constructs a one-to-one correspondence between dense subcategories of $\C$ and subgroups of $G_0(\C)$ as follows:

To $\D$ a dense subcategory of $\C$ corresponds the subgroup which is the image of $G_0(\D)$ in $G_0(\C)$. To  $H$ a subgroup of $G_0(\C)$
corresponds the full subcategory $\D_H$ whose objects are those $X$ in $\C$ such that
$[X] \in H$.

If $\D$ is a dense subcategory of $\C$ then the natural map $G(\D) \rt G(\C)$ is injective, see \cite[2.3]{T}.

\s\label{elt-Q} Let $e \in G(\C)_\Q$ then $e =\frac{1}{m}[U] $ for some $m \geq 1$ and $U \in \C$. To see this we note that
$$ e = \frac{a_1}{b_1}[U_1] + \cdots + \frac{a_s}{b_s}[U_s] \quad \text{for some $U_i \in \C$ and $a_i \in \Z$ non-zero and $b_i \geq 1$}.$$
Let $m$ be the l.c.m of the $b_i$.
Then
\begin{align*}
 e &= \frac{1}{m}\left( a_1^\prime [U_1] + \cdots + a_s^\prime[U_s] \right), \\
  &= \frac{1}{m}\left(  [V_1] + \cdots + [V_s] \right),   \  \text{here $V_i = U_i^{a_i^\prime}$ if $a_i^\prime > 0$ and $V_i = U_i[-1]^{-a_i^\prime}$ otherwise}, \\
   &= \frac{1}{m}[V] \quad \text{where $V = V_1\oplus \cdots \oplus V_s$}.
\end{align*}

\begin{lemma}\label{exist}
Let $H$ be a $\Q$-subspace of $G(\C)_\Q$. Then there exists a dense subcategory of $\C$ with $G(\D)_\Q = H$.
\end{lemma}
\begin{proof}
Let $\{e_\alpha \}_{\alpha \in \Lambda} $ be a basis of $H$ as a $\Q$-vector space.
 By \ref{elt-Q} we have $e_\alpha = \frac{1}{m_\alpha}[U_\alpha]$ for some positive integer $m_\alpha$ and $U_\alpha \in \C$. Set $$v_\alpha = m_\alpha e_\alpha = \frac{1}{1}[U_\alpha]. \quad \text{for all $\alpha \in \Lambda$}. $$
 Then $\{v_\alpha \}_{\alpha \in \Lambda} $ is also a basis of $H$ as a $\Q$-vector space. Let $K$ be the ($\Z$)-subgroup of $G(\C)$ generated by $\{ [U_\alpha] \}_{\alpha \in \Lambda}$. Note that $K$ is free as a $\Z$-module and $K_\Q = H$. Let $\D$ be the dense subcategory of $\C$ corresponding to $K$. Then $G(\D) = K$ and so $G(\D)_\Q = H$.
\end{proof}

\s \label{const} \emph{Construction :} Let $H$ be a $\Q$-subspace of $G(\C)_\Q$.
If $\D$ is a dense subcategory of $\C$ then as we have observed earlier the natural map $\delta^\D \colon G(\D) \rt G(\C)$ is an injection. So $\delta^\D_\Q \colon G(\D)_\Q \rt G(\C)_\Q$ is an injection
Let $$\I(H) = \{ \D \mid \D \ \text{is a dense subcategory of $\C$ such that $\image \delta^\D_\Q = H$} \}.$$
By \ref{exist} we get that $\I(H)$ is non-empty. We define a partial order $\leq$ on $\I(D)$ as $\D \leq \D^\prime$ if $\D \subseteq \D^\prime$.
\begin{lemma}
\label{max}(with hypotheses as in \ref{const}). There exists maximal elements in $\I(H)$ with respect to the partial order $\leq$.
\end{lemma}
\begin{proof}
  The result will follow from Zorn's lemma if we prove every chain in $\I(H)$ has an upper bound in $\I(H)$.

  Let $\{ \D_\alpha \}_{\alpha \in \Lambda}$ be a chain in $\I(H)$. Then it is elementary to see that $\D = \bigcup_{\alpha \in \Lambda}\D_\alpha$ is a dense subcategory of $\C$.

 Claim: $\D \in \I(H)$.

 Note clearly $\D$ is an upper bound of $\{ \D_\alpha \}_{\alpha \in \Lambda}$. So it suffices to prove the claim.

  If $\D_\alpha \subseteq D_\beta $ then note that as $\D_\alpha$ is dense in $\C$ it is also a dense subcategory of $\D_\beta$. As observed before the natural map
  $i^{\D_\alpha}_{\D_\beta} \colon  G(\D_\alpha) \rt G(\D_\beta)$ is an injection. We have a commutative diagram
    \[
\xymatrix{
\
&G(\D_\alpha)
\ar@{->}[dl]_{i^{\D_\alpha}_{\D_\beta}}
\ar@{->}[dr]^{\delta^{\D_\alpha}}
 \\
G(\D_\beta)
\ar@{->}[rr]_{  \delta^{\D_\beta }}
&\
&G(\C)
}
\]
So we have  $\image \delta^{\D_\alpha} \subseteq \image \delta^{\D_\beta}$. It follows that $\bigcup_{\alpha \in \Lambda} \image \delta^{\D_\alpha}$ is a subgroup of $G(\C)$. It is also clearly equal to $\image \delta^\D$. We have
  $\lim_{\alpha \in \Lambda}G(\D_\alpha) = G(\D)$. It follows that $\lim_{\alpha \in \Lambda}G(\D_\alpha)_\Q = G(\D)_\Q$. For $\D_\alpha \subseteq \D_\beta$, we have a commutative diagram
  \[
\xymatrix{
\
&G(\D_\alpha)_\Q
\ar@{->}[dl]_{i^{\D_\alpha}_{\D_\beta}\otimes \Q}
\ar@{->}[dr]^{\delta^{\D_\alpha}_\Q}
 \\
G(\D_\beta)_\Q
\ar@{->}[rr]_{  \delta^{\D_\beta}_\Q }
&\
&G(\C)_\Q
}
\]
All the maps above are injections. As $\image \delta^{\D_\alpha}_\Q = \image \delta^{\D_\beta}_\Q = H$, it follows that $i^{\D_\alpha}_{\D_\beta}\otimes \Q$ is an isomorphism. It follows that $G(\D)_\Q \cong G(\D_\alpha)_\Q$ for all $\alpha \in \Lambda$. It also follows that $\image \delta^\D_\Q = H$. Thus $\D \in \I(H)$.
\end{proof}

\s For definition of radical dense subcategories see \ref{def-rad}. We prove Proposition \ref{sqrt}. For the convenience of the reader we re-state it here.
\begin{proposition}\label{sqrt-body}
Let  $\D$ be a dense subcategory of $\C$. Let
$$\sqrt{\D} = \{ U \in \C \mid U^n \in \D \ \text{for some $n \geq 1$} \}. $$
Then
\begin{enumerate}[\rm (1)]
\item $\D \subseteq \sqrt{\D}$.
  \item $\sqrt{\sqrt{D}} = \sqrt{\D}$.
  \item $\sqrt{\D}$ is a radical dense subcategory of $\C$.
   \item $\D$ is a radical dense subcategory of $\C$ if and only if $\sqrt{\D} = \D$.
\end{enumerate}
\end{proposition}
\begin{proof}
  (1) This is clear.

  (2) Let $U \in \sqrt{\sqrt{D}}$. Then $U^m \in \sqrt{\D}$. It follows that $(U^m)^s \in \D$. So $U^{ms} \in \D$. Thus $U \in \sqrt{\D}$.

  (3) We first prove $\sqrt{\D}$ is a triangulated subcategory of $\C$. The conditions (1) and (3) of \ref{sub-tri} are trivially satisfied.
  Let $U \xrightarrow{f} V \rt W \rt U[1]$ be a triangle in $\C$ with $U, V \in \sqrt{\D}$. We have to prove $W \in \sqrt{\C}$.
  As $U, V \in \sqrt{\D}$ we get $U^m, V^n \in \D$ for some $m, n$. Then $U^r, V^r \in \D$ (here $r = mn$). Consider the map
  \[
  U^r \xrightarrow{ \phi = \begin{bmatrix}
                      f & 0 & 0 & \cdots & 0 \\
                      0 & f & 0 & \cdots & 0 \\
                      \cdots & \cdots & \cdots & \cdots & \cdots \\
                      0 & 0 & 0 & \cdots & f
                    \end{bmatrix}}
                    V^r.
  \]
  By \cite[1.2.1]{N} we have a triangle
  \[
  U^r \xrightarrow{\phi} V^r \rt W^r \rt U^r[1].
  \]
  It follows that $W^r \in \D$. So $W \in \sqrt{\D}$. Thus $\sqrt{\D}$ is a triangulated sub-category of $\C$.

  As $\D$ is dense in $\C$ and as $\D \subseteq \sqrt{\D}$ it follows that $\sqrt{\D}$ is dense in $\C$.

  (4) If $\sqrt{\D} = \D$ then by (3) we have that $\D$ is a radical dense subcategory of $\C$.
  Conversely if $\D$ is a radical dense subcategory of $\C$ then by definition $\sqrt{\D} = \D$.
\end{proof}
The following result is significant in our analysis of subspaces of $G(\C)_\Q$.
\begin{lemma}\label{equal}
(with hypotheses as in \ref{sqrt-body}) The natural map $G(\D)_\Q \rt G(\sqrt{\D})_\Q$ is an isomorphism.
\end{lemma}
\begin{proof}
  As $\D$ is dense in $\C$ it is trivially dense in $\sqrt{\D}$. So the natural map $i \colon G(\D) \rt G(\sqrt{\D})$ is an injection, Thus $i_\Q$ is injective.
  Let $\xi \in G(\sqrt{D})$. Then by \ref{elt-Q} we have $\xi = \frac{1}{m}[U]$ for some $U \in \sqrt{\D}$ and $m \geq 1$. As $U\in \sqrt{\D}$ we get that $U^r \in \D$ for some $r \geq 1$.
Observe that
$$ i_\Q(\frac{1}{mr}[U^r]) = \xi. $$
Thus $i_\Q$ is surjective. The result follows.
\end{proof}

As a consequence of \ref{equal} we obtain
\begin{lemma}
\label{max-dense}(with hypotheses as in \ref{const}). Then maximal elements in $\I(H)$ with respect to the partial order $\leq$ are radical dense subcategories.
\end{lemma}
The following trivial observation is crucial.
\begin{lemma}
\label{torsion} Let $\C$ be an essentially small triangulated category and let $\D$ be a radical dense subcategory of $\C$. If $[U] \in G(\C)$ is a torsion element then $U \in \D$.
\end{lemma}
\begin{proof}
  Let $\D$ correspond to the subgroup $H$ of $G(\C)$. So $\D$ consists of all elements $V$ such that $[V] \in H$.

  As $[U]$ is a torsion element we have $n[U] = 0$ for some $n \geq 2$. We have $[U^n] = n[U] = 0$. In particular $[U^n] \in H$. So $U^n \in \D$.
  As $\D$ is radical we get that $U \in \D$.
\end{proof}
In \ref{max} we proved that if $H$ is a subgroup of $G(\C)$ then $\I(H)$ has maximal elements. Our next result further yields the structure of $\I(H)$.
\begin{proposition}
(with hypotheses as in \ref{max}). There is a unique maximal element in $\I(H)$.
\end{proposition}
\begin{proof}
  Let $\D_1$ and $\D_2$ be maximal elements in $\I(H)$. Let $U \in \D_1$. Then $\frac{1}{1}[U] \in H = \image \delta^{\D_2}_\Q$. By \ref{elt-Q} we get that
  $$\frac{1}{1}[U] = \frac{1}{m}[V] \quad \text{for some $V \in \D_2$ and $m \geq 1$.}$$
  It follows that $[U^m] -[V] = [W]$ where $[W]$ is a torsion element in $G(\C)$. As $\D_2$ is a radical dense subcategory of $\C$ we get that $W \in \D_2$, see \ref{torsion}. So $[U^m] \in \image \delta^{\D_2}$. Thus $U^m \in \D_2$. Again as $\D_2$ is a radical dense subcategory of $\C$ we get that $U \in \D_2$. Thus $\D_1 \subseteq \D_2$. Similarly $\D_2 \subseteq \D_1$. So $\D_1 = \D_2$.
\end{proof}

\s So far we have proved that there is a unique maximal element of $\I(H)$ which is a radical  dense subcategory of $\C$. Next we show that there is a unique radical dense sub-category in $\I(H)$.
More precisely we show
\begin{proposition}\label{rad-unique}
 Let $\C$ be an essentially small triangulated category and let $\B$ be a radical dense subcategory of $\C$. Let $H = \image(\delta_\Q^\B \colon G(\B) \rt G(\C))$. Let $\D$ be the unique maximal element in $\I(H)$. Then $\B = \D$.
\end{proposition}
\begin{proof}
We note that $\B \in \I(H)$. By uniqueness of maximal element in $\I(H)$ we get that $\B \subseteq \D$.
Let $U \in \D$. Then $[U] \in H = \image \delta^{\B}_\Q$. By \ref{elt-Q} we get that
  $$\frac{1}{1}[U] = \frac{1}{m}[V] \quad \text{for some $V \in \B$ and $m \geq 1$.}$$
   It follows that $[U^m] -[V] = [W]$ where $[W]$ is a torsion element in $G(\C)$. As $\B$ is a radical dense subcategory of $\C$ we get that $W \in \B$, see \ref{torsion}. So $[U^m] \in \image \delta^{\B}$. Thus $U^m \in \B$. Again as $\B$ is a radical dense subcategory of $\C$ we get that $U \in \B$. Thus $\D \subseteq \B$. So $\D = \B$.
\end{proof}

We describe a group theoretic criterion for a dense category of $\C$ to be radical.
\begin{proposition}\label{rad-quot}
 Let $\C$ be an essentially small triangulated category and let $\D$ be a  dense subcategory of $\C$.
 The following assertions are equivalent
 \begin{enumerate}[\rm (i)]
 \item
 $\D$ is radical.
 \item
 $G(\C)/G(\D)$ is torsion free.
 \end{enumerate}
\end{proposition}
\begin{proof}
(i) $\implies$ (ii): Let $\ov{[U]}$ be a torsion element in $G(\C)/G(\D)$. Then $n[U] \in G(\D)$. So $U^n \in \D$. As $\D$ is radical we get that $U \in \D$. So $\ov{[U]} = 0$.

(ii) $\implies$ (i): Suppose $U^n \in \D$. Then note that $\ov{[U]}$ is a torsion element in $G(\C)/G(\D)$. By assumption $G(\C)/G(\D)$ is torsion free. So $\ov{[U]} = 0$. Therefore $[U] \in G(\D)$.  Thus $U \in \D$. So $\D$ is radical.
\end{proof}
Finally we prove Theorem \ref{main-cat}. We restate it here for the convenience of the reader.
\begin{theorem}\label{main-cat-body}
Let $\C$ be a skeletally small triangulated category. Let $H$ be a $\Q$-subspace of $G(\C)_\Q$. Then
\begin{enumerate}[\rm (1)]
  \item There exists a radical  dense subcategory $\D$ of $\C$ such that $\image \delta^\D_\Q = H$.
  \item If $\D_1, \D_2$ are two radical dense subcategories of $\C$ such that $H = \image \delta^{\D_1}_\Q = \image \delta^{\D_2}_\Q$ then $\D_1 = \D_2$. Thus we may write $\D_H$ to be the unique radical  dense subcategory of $\C$ with $\image \delta^\D_\Q = H$.
\item Let $H_1, H_2$ be $\Q$-subspaces of $G(\C)_\Q$. We have $\D_{H_1} \subseteq \D_{H_2}$ if and only if $H_1 \subseteq H_2$.
\end{enumerate}
\end{theorem}
\begin{proof}
(1) This follows from \ref{max-dense}.

(2) This follows from \ref{rad-unique}.

(3) First assume $\D_1 \subseteq \D_2$. Note that as $\D_1$ is dense in $\C$ it is also a dense subcategory of $\D_2$. As observed before the natural map
  $i^{\D_1}_{\D_2} \colon  G(\D_1) \rt G(\D_2)$ is an injection. We have a commutative diagram
    \[
\xymatrix{
\
&G(\D_1)
\ar@{->}[dl]_{i^{\D_1}_{\D_2}}
\ar@{->}[dr]^{\delta^{\D_1}}
 \\
G(\D_2)
\ar@{->}[rr]_{  \delta^{\D_2 }}
&\
&G(\C)
}
\]
After tensoring with $\Q$ we have a commutative diagram
  \[
\xymatrix{
\
&G(\D_1)_\Q
\ar@{->}[dl]_{i^{\D_1}_{\D_2}\otimes \Q}
\ar@{->}[dr]^{\delta^{\D_1}_\Q}
 \\
G(\D_2)_\Q
\ar@{->}[rr]_{  \delta^{\D_2}_\Q }
&\
&G(\C)_\Q
}
\]
All the maps above are injections. So $H_1 = \image \delta^{\D_1}_\Q \subseteq \image \delta^{\D_2}_\Q  = H_2$.

Conversely assume $H_1 \subseteq H_2$. Let $\D_{H_2}$ be the radical dense subcategory of $\C$ corresponding to $H_2$. We note that $G(\D_{H_2})_\Q = H_2$. As $H_1 $ is $\Q$-subspace of $G(\D_2)_\Q$ by (1) there exists a radical dense subcategory $\D_1$ of $\D_{H_2}$ such that $(i^{\D_1}_{\D_{H_2}}\otimes \Q)(G(\D_1)_\Q) = H_1$. Observe that from the above commutative diagram we have $\image \delta^{\D_1}_\Q = H_1$. Note that $\D_1$ is also a dense and radical subcategory of $\C$. By (2) it follows that $\D_1 = \D_{H_1}$. The result follows.
\end{proof}

\section{Proof of Theorem \ref{main} and Theorem \ref{main-st}}
In this section we give proofs of Theorem \ref{main} and Theorem \ref{main-st}. We first prove:
\begin{theorem}\label{main3-body}
Let $(A,\m)$ be an excellent Gorenstein isolated singularity of dimension $d \geq 2$. Let $\eta \colon G(A) \rt G(\wh{A})$ and $\theta \colon G(\CMS(A)) \rt G(\CMS(\wh{A}))$ be the natural maps.
Then $\eta$ and $\theta$ are injective.
The following assertions are equivalent:
\begin{enumerate}[\rm (i)]
  \item $\eta_\Q$ is an isomorphism.
  \item $\theta_\Q$ is an isomorphism.
\end{enumerate}
\end{theorem}
\begin{proof}
We first note that $A$ and $\wh{A}$ are domains.
So the natural map $i \colon \Z \rt G(A)$  given by $1 \mapsto [A]$ is split by the map $r  \colon G(A) \rt  \Z$ defined by $r([X]) = \rank(X)$.
We have $G(\CMS(A)) = G(A)/\Z [A]$. We have a commutative diagram
 \[
  \xymatrix
{
 0
 \ar@{->}[r]
  & \Z
\ar@{->}[r]^{i}
\ar@{->}[d]^{1}
 & G(A)
\ar@{->}[r]^{\pi}
\ar@{->}[d]^{\eta}
& G(\CMS(A))
\ar@{->}[r]
\ar@{->}[d]^{\theta}
&0
\\
 0
 \ar@{->}[r]
  & \Z
\ar@{->}[r]^{i^\prime}
 & G(\wh{A})
\ar@{->}[r]^{\pi^\prime}
& G(\CMS(\wh{A}))
    \ar@{->}[r]
    &0
\
 }
\]
It follows that $\eta$ is injective if and only if $\theta$ is injective.

Let $\psi \colon \CMS(A) \rt \CMS(\wh{A})$ be the obvious functor (which is clearly a triangulated functor).
As $A$ is an isolated singularity we have
$\sHom_A(M, N)$ has finite length for every $M, N \in \CMS(A)$. It follows that $\psi$  is  fully faithful. Furthermore   every MCM $\wh{A}$-module $M$ is a direct summand of $N \otimes_A \wh{A}$ where $N$ is a MCM
$A$-module, implicitly this is in \cite[2.9]{W}; for an explicit proof see \cite[3.2]{Ta}. By Theorem \ref{b-1} the natural map
$\theta \colon G(\CMS(A)) \rt  G(\CMS(\wh{A}))$ is injective.
Thus $\eta$ is also injective.

From the above commutative diagram it follows that $\eta_\Q$ is an isomorphism if and only if $\theta_\Q$ is an isomorphism.
\end{proof}

Next we prove Theorem \ref{main-st}. We restate it here for the convenience of the reader.
\begin{theorem}\label{main-st-body}
(with hypotheses as in \ref{stable}). The following assertions are equivalent:
\begin{enumerate}[\rm (i)]
  \item $\theta_\Q$ is an isomorphism.
  \item For any MCM $\wh{A}$-module  $M$ there exists an MCM $A$-module $N$ and integer $r \geq 1$ (depending on $M$) such that $M^r\cong \wh{N}$ in $\CMS(\wh{A})$.
\end{enumerate}
\end{theorem}
\begin{proof}
We first prove
(ii) $\implies$ (i):  By Theorem \ref{main3-body} it follows that $\theta_\Q$ is injective. Let $\xi \in G(\CMS(\wh{A}))_\Q$. By \ref{elt-Q} we get $\xi = \frac{1}{m}[M]$ for some $m \geq 1$ and a MCM $\wh{A}$-module $M$. By hypothesis there exists a MCM $A$-module with $M^r\cong \wh{N}$ in $\CMS(\wh{A})$. We note that $\theta_\Q(\frac{1}{mr}[N]) = \xi$. So $\theta_\Q$ is surjective.

Next we prove
(i) $\implies$ (ii): We have $\theta \colon G(\CMS(A)) \rt G(\CMS(\wh{A}))$ is an injection. Let $H = \image \theta$ and let $\D$ be the Thomason dense subcategory associated with $H$. We have
$\image \delta^{\sqrt{\D}}  = H_\Q$. By hypotheses we have
$H_\Q = G(\CMS(\wh{A}))$. So by Theorem \ref{main-cat-body} it follows that $\sqrt{\D} = \CMS(\wh{A})$.  Thus if $M \in \CMS(\wh{A})$ then $M^r \in \D$ for some $r \geq 1$. So $[M^r]$ is in the image of $\theta$. As $\CMS(\wh{A})$ satisfies weak cancellation (as it is Krull-Schmidt) it follows from \ref{b-onto} that $M^r \cong \wh{N}$ in $\CMS(\wh{A})$  for some MCM $A$-module $N$.
\end{proof}
Finally we give a proof of
Theorem \ref{main}. We restate it here for the convenience of the reader.
\begin{theorem}\label{main-body}
Let $(A,\m)$ be an excellent Gorenstein isolated singularity of dimension $d \geq 2$. Let $\eta \colon G(A) \rt G(\wh{A})$ be the natural map. The following assertions are equivalent:
\begin{enumerate}[\rm (i)]
  \item $\eta_\Q$ is an isomorphism.
  \item For any MCM $\wh{A}$-module  $M$ there exists an MCM $A$-module $N$ and integers $r \geq 1$ and $s \geq 0$ (depending on $M$) such that $M^r\oplus \wh{A}^s \cong \wh{N}$.
\end{enumerate}
\end{theorem}
\begin{proof}
  (i) $\implies$ (ii): If $\eta_\Q$ is an isomorphism then by \ref{main3-body} we get that $\theta_\Q$ is an isomorphism. So by \ref{main-st-body} if $M$ is a MCM $\wh{A}$-module then there exists an MCM $A$-module $N_1$ and $r \geq 1$ such that $M^r \cong \wh{N_1}$ in $\CMS(\wh{A})$. As $\wh{A}$-modules  we get $M \oplus \wh{A}^s = \wh{N_1} \oplus \wh{A}^t$ for some $s, t \geq 0$. Set $N = N_1\oplus A^t$. The result follows.

  (ii) $\implies$ (i): By our hypotheses it follows that for any MCM $\wh{A}$-module  $M$ there exists an MCM $A$-module $N$ and integer $r \geq 1$ such that $M^r\cong \wh{N}$ in $\CMS(\wh{A})$.
  So by \ref{main-st-body}, $\theta_\Q$ is an isomorphism.  By \ref{main3-body} we get that $\eta_\Q$ is an isomorphism.
\end{proof}

\begin{remark}
\label{G-ext} One can define Grothendieck group of any extension closed additive subcategory of $\mmod(A)$. Let $\CMa(A)$ denote the additive category of MCM $A$-modules. The natural map
$G(\CMa(A))\rt G(A)$ is an isomorphism, see \cite[13.2]{Y}. Using this the assertion  (ii) $\implies$ (i) is trivial (we have to use $\eta_\Q$ is injective).
\end{remark}
\section{Proof of Theorem \ref{sub-ver}}
In this section we give a proof of Theorem \ref{sub-ver}. We prove it very generally.

\s \label{hyp-gen} We consider pairs $\D, \wh{\D}$ of thick subcategories of $\C$ and $\wh{C}$ respectively with the following properties:
\begin{enumerate}
\item We have a triangulated functor $\psi \colon \C \rt \wh{\C}$ which is an equivalence up to direct summands.
 \item $\wh{\C}$ has weak cancellation.
  \item If $M \in \D$ then $\psi(M) \in \wh{\D}$
  \item If $\psi(M) \in \wh{\D}$ then $M \in \D$.
\end{enumerate}
Let $\T = \C/\D$ and $\wh{\T} = \wh{\C}/\wh{\D}$. the Verdier quotient. We have natural maps $\theta \colon G(\C) \rt G(\wh{\C})$, $\alpha \colon G(\D) \rt G(\wh{\D})$ and $\beta \colon G(\T) \rt G(\wh{\T})$. We first show
\begin{proposition}\label{first} (with hypotheses as in \ref{hyp-gen}) We have
\begin{enumerate}[\rm (1)]
   \item $\theta$ is an injection.
  \item  The induced functor $\psi^\sharp \colon \D \rt \wh{\D}$ is an equivalence up to direct summands.
  \item The induced functor $\ov{\psi} \colon \T \rt \wh{\T}$ is an equivalence up to direct summands.
  \item $\alpha$ and $\beta$ are injections.
\end{enumerate}
\end{proposition}
\begin{proof}
(1) This follows from \ref{b-1}.

(2) Clearly $\psi^\sharp$ is fully faithful. Also if $U \in \wh{D}$ then note that the considered as an element of $G(\wh{\C})$ we have $[U \oplus U[1]] = 0$. By hypothesis $\wh{C}$ has weak cancellation. So by \ref{b-onto} there exists $V \in \C$ with $\psi(V) = U\oplus U[1]$. As $\psi(V) \in \wh{\D}$ it follows by our hypotheses that $V \in \D$. The result follows.

(3) We first show that $\ov{\psi}$ is fully faithful. Let $U, V \in \T$. Set $\wh{U}, \wh{V}$ their images in $\wh{T}$.  We have to show the natural map
$\delta \colon \Hom_{\T}(U, V) \rt \Hom_{\wh{\T}}(\wh{U}, \wh{V})$ is bijective.

 We first show that $\ov{\psi}$ is faithful. Let $\xi \in \Hom_{\T}(U, V)$ with $\delta(\xi) = 0$.
We write $\xi$ as a left fraction
\[
\xymatrix{
\
&E
\ar@{->}[dl]_{g}
\ar@{->}[dr]^{f}
 \\
U
\ar@{->}[rr]_{\xi = fg^{-1}}
&\
&V
}
\]
We have $0 = \delta(\xi) = \delta(f) \circ \delta(g^{-1}) = \delta(f) \circ \delta(g)^{-1}$. So $\delta(f) = 0$. It follows that $\delta(f)$ factors through an element $W \in \wh{D}$, see \cite[2.1.26]{N}.
Say $\delta(f) = r\circ s$ where $s \colon  \wh{E} \rt W$ and $r \colon W \rt \wh{V}$. By (2) above we get that $W\oplus W[1] = \wh{Z}$ for some $Z \in \D$. Set $s^\prime \colon \wh{E} \rt W\oplus W[1] $ where $s^\prime  = (s, 0)$ and $r^\prime \colon W \oplus W[1]\rt \wh{V}$  as $r^\prime = (r, 0)$. Then $\delta(f) = r^\prime \circ s^\prime$. As $\psi$ is fully faithful it follows that $f$ factors through $Z$. So $f = 0$ in $\T$. Therefore $\xi = 0$. Thus $\ov{\psi}$ is faithful.

Next we show $\ov{\psi}$ is full. Suppose we have $\xi \in \Hom_{\wh{\T}}(\wh{U}, \wh{V})$.
We write $\xi$ as a left fraction
\[
\xymatrix{
\
&E
\ar@{->}[dl]_{g}
\ar@{->}[dr]^{f}
 \\
\wh{U}
\ar@{->}[rr]_{\xi = fg^{-1}}
&\
&\wh{V}
}
\]
We have a triangle $E \rt \wh{U} \rt W \rt E[1]$ with $W\in \wh{\D}$. Rotating we have a triangle $s \colon W[-1] \rt E \rt \wh{U} \rt W$. By (2) there exists $L \in \D$ with $\wh{L} = W[-1]\oplus W$.
Adding $W \xrightarrow{1}W \rt 0 \rt W[1]$ to $s$ we obtain a triangle $e \colon W[-1]\oplus W \rt E\oplus W \rt \wh{U} \rt W \oplus W[1]$. Rotating we have a triangle
$l \colon \wh{U[-1]} \xrightarrow{\eta} W[-1] \oplus W \rt E \oplus W \rt \wh{U}$. As $\psi$ is fully faithful we get $\eta = \psi(\delta)$ where $\delta \colon U[-1] \rt L$. It follows that
$E\oplus W = \psi(E^\prime)$ for some $E^\prime \in \C$.
We note that the natural map $\pi \colon E\oplus W \rt E$ is invertible in $\wh{T}$. So we have a left fraction
\[
\xymatrix{
\
&E\oplus W
\ar@{->}[dl]_{g\circ \pi}
\ar@{->}[dr]^{f \circ \pi}
 \\
\wh{U}
\ar@{->}[rr]_{\xi = fg^{-1} = (f\circ \pi)\circ (g \circ \pi)^{-1}}
&\
&\wh{V}
}
\]
As $\psi$ is fully faithful $g \circ \pi = \psi(g^\prime)$ where  $g^\prime \colon E^\prime \rt U$ and $f\circ \pi = \psi(f^\prime)$  where $f^\prime \colon E^\prime \rt V$. As $\cone(g \circ \pi) \in \wh{\D}$ it follows that $\cone(g^\prime) \in \D$. It follows that $\xi = \psi(\xi^\prime)$ where $\xi^\prime = f^\prime \circ (g^\prime)^{-1} \in \Hom_\T(U, V)$. Thus $\ov{\psi}$ is full.

Let $U \in \wh{T}$. Considering $U$ as an object in $\wh{C}$ there exists $V \in \C$ such that $U$ is a direct summand of $\psi(V)$. By considering their images we get that there exists $V$ in $\T$ such that $U$ is a direct summand of $\ov{\psi}(V)$.

Thus the functor $\ov{\psi} \colon \T \rt \wh{\T}$ is an equivalence up to direct summands.

(4) This follows from (2) and (3) and \ref{b-1}.
\end{proof}
Next we show
\begin{theorem}
  \label{abc} (with hypotheses as in \ref{first}). We have an exact sequence
  $$ 0 \rt \coker(\alpha) \rt \coker(\theta) \rt \coker(\beta) \rt 0.$$
\end{theorem}
As an immediate corollary we obtain
\begin{corollary}
(with hypotheses as in \ref{abc})\begin{enumerate}[\rm (1)]
                                   \item $\theta$ is an isomorphism if and only if $\alpha$, $\beta$ are isomorphisms
                                   \item $\theta_\Q$ is an isomorphism if and only if $\alpha_\Q$, $\beta_\Q$ are isomorphisms
                                 \end{enumerate}
\end{corollary}
\begin{proof}
We use the fact that $\alpha, \beta$ and $\theta$ are injective, see Proposition \ref{first},  and Theorem \ref{abc}.
\end{proof}
Next we give
\begin{proof}[Proof of Theorem \ref{abc}]
By \ref{Verd}
 we have a commutative diagram
 \[
  \xymatrix
{
 \
 \
  & G(\D)
\ar@{->}[r]^{i}
\ar@{->}[d]^{\alpha}
 & G(\C)
\ar@{->}[r]
\ar@{->}[d]^{\theta}
& G(\T)
\ar@{->}[r]
\ar@{->}[d]^{\beta}
&0
\\
 \
 \
  & G(\wh{\D})
\ar@{->}[r]^{\wh{i}}
 & G(\wh{\C})
\ar@{->}[r]
& G(\wh{\T})
    \ar@{->}[r]
    &0
\
 }
\]
Let $K = \ker i$ and $\wh{K} = \ker \wh{i}$. We note that $\alpha$ and $\theta$ induce a map $\delta \colon K \rt \wh{K}$ which is an injection (since $\alpha$ is an injection). So we have a commutative diagram
\[
  \xymatrix
{
 0
 \ar@{->}[r]
 & K
\ar@{->}[r]^{j}
\ar@{->}[d]^{\delta}
  & G(\D)
\ar@{->}[r]^{i}
\ar@{->}[d]^{\alpha}
 & G(\C)
\ar@{->}[r]
\ar@{->}[d]^{\theta}
& G(\T)
\ar@{->}[r]
\ar@{->}[d]^{\beta}
&0
\\
 0
 \ar@{->}[r]
  & \wh{K}
\ar@{->}[r]^{\wh{j}}
  & G(\wh{\D})
\ar@{->}[r]^{\wh{i}}
 & G(\wh{\C})
\ar@{->}[r]
& G(\wh{\T})
    \ar@{->}[r]
    &0
\
 }
\]
Claim: $\delta$ is an isomorphism.

Note if we prove the Claim then the result follows by a routine diagram chase.

Proof of Claim: We already know that $\delta$ is injective.
We prove that $\delta$ is also surjective. We consider $K$ and $\wh{K}$ as subgroups of $G(\D)$ and $G(\wh{\D})$ respectively. Let $[V] \in \wh{K}$. Then $\wh{i}([V] = 0$. We note that $\wh{i}([V]) = [V]$ considered as an element of $G(\wh{\C})$. So $[V] = 0 \in \image(\theta)$. Also by hypothesis $\wh{C}$ satisfies weak cancellation. So  by \ref{b-onto} there exists $U \in \C$ such that $\psi(U) = V$. By our hypotheses $U \in \D$. Also $i([U]) = 0$ (since $\theta$ is injective). So $[U] \in K$. Clearly $\delta([U]) = [V]$. So $\delta$ is surjective.
\end{proof}
\s\label{cx-curv} We give the definition of complexity and curvature. We follow \cite{A}. Let $(R, \n)$ be a Noetherian local ring and let $M$ be a finitely generated  $R$-module. Let $\beta_n(M)$ denote the $n^{th}$-betti number of $M$. Define
$$\cx_R M = \inf\{ i \mid \limsup_{n \rt \infty} \beta_n(M)/n^{i-1} < \infty \}.$$
If $(A,\m)$ is a complete intersection of codimension $c$ then $\cx_A M \leq c$. Also for $i = 0, \ldots, c$ there exists an $A$-module $M$ with complexity $i$.
It is elementary to verify that if $A$ is a complete intersection then $\D_i$ the set of MCM modules $M$ with complexity $\leq i$ is a thick subcategory of $\CMS(A)$.

If $R$ is not a complete intesection then $\cx_R k =\infty$ (here $k$ is the residue field of $R$). To deal with this the notion of curvature was introduced.
Set
$$\curv_R M = \limsup_{n \rt \infty} \sqrt[n]{\beta_n(M)}. $$
It can be shown that $\curv M \leq \curv k < \infty$. Also if $R$ is not a complete intersection then $\curv k > 1$. I{f $A$ is a Gorenstein ring then for $1 < \beta \leq \curv k$, it is easy to verify that $\D_\beta$,
the set of MCM $A$-modules with complexity $< \beta$ is a thick subcategory of $\CMS(A)$.

Finally we give
\begin{proof}[Proof of Theorem \ref{sub-ver}]
It can be easily verified that $\D, \CMS(A), \wh{D}, \CMS(\wh{A})$ satisfy the hypotheses of
\ref{hyp-gen}.
\end{proof}

\section{proof of Theorem \ref{chow-1}}
We first give
\begin{proof}[ Proof of Theorem \ref{chow-1}]
By \cite[1.5]{KK}, the natural map $G(A) \rt G(\wh{A})$ is injective. We have $G(A)_\Q \cong A_*(A)_\Q$. It follows from \cite[2.4]{KK} that the natural maps $A_i(A) \rt A_i(\wh{A})$ are injective for all $i \geq 0$.
 We note that $F_i(A) = <[A/Q] \mid \dim A/Q \leq i>$ defines a filtration on $G(A)$. Set $G_i(A)_\Q = (F_i(A)/F_{i-1}(A))_\Q$.
  We also have  $G(A)_\Q \cong \bigoplus_{i \geq 0} G_i(A)_\Q$. By Riemann-Roch we have $G(A)_\Q \cong A_*(A)_\Q$.
 The Riemann–Roch map decomposes into isomorphisms $G_i(A)_\Q \cong A_i(A)_\Q$ for $i \geq 0$.
 We show the natural map  $G_1(A)_\Q \rt G_1(\wh{A})_\Q$ is an isomorphism.

We have $G(\CMS(A)) = G(A)/\Z [A]$. We have a commutative diagram (as $A, \wh{A}$ are domains)
 \[
  \xymatrix
{
 0
 \ar@{->}[r]
  & \Q
\ar@{->}[r]^{i}
\ar@{->}[d]^{1}
 & G(A)_\Q
\ar@{->}[r]^{\pi}
\ar@{->}[d]^{\eta_\Q}
& G(\CMS(A))_\Q
\ar@{->}[r]
\ar@{->}[d]^{\theta_\Q}
&0
\\
 0
 \ar@{->}[r]
  & \Q
\ar@{->}[r]^{i^\prime}
 & G(\wh{A})_\Q
\ar@{->}[r]^{\pi^\prime}
& G(\CMS(\wh{A}))_\Q
    \ar@{->}[r]
    &0
\
 }
\]
As $A$ is excellent and as $A$ satisfies $R_{d-2}$ it follows that  $\wh{A}$ satisfies $R_{d-2}$, \cite[23.9]{Ma}.
We may assume that $A$ is not an isolated singularity as otherwise we have nothing to prove.
Let $I = P_1\cap\cdots \cap P_r$ define the singular locus of $\wh{A}$. Then $\height P_i = d-1$ for all $i$.

Let $Q$ be a prime of height $d-1$ in $\wh{A}$.

Case-1: $Q \neq P_j$ for all $j$. \\
Then $\Syz^A_{d}(A/Q) = X$ is free on the punctured spectrum of $A$. So there exists an MCM $A$-module $Y$ with $\wh{Y} \cong X$, see \cite[Theorem 3]{E}. We note that the image of $[\wh{A}/Q]$ in  $G(\CMS(\wh{A}))$
is $(-1)^d[X]$. So $(-1)^d[Y] \in G(\CMS(A))$ maps to the image of $[A/Q]$ in $\CMS(\wh{A})$. A diagram chase shows that there exists $[U] \in G(A)_\Q$ mapping to $[\wh{A}/Q] \in G(\wh{A})_\Q$.
 If $[\wh{A}/Q] \neq 0 $ in $G(\wh{A})_Q$ then $[U] \neq 0$ in $G(A)_\Q$. As $G(A)_\Q = \bigoplus_{i \geq 0} G_i(A)_\Q$ and the injective map $G(A)_\Q \rt G(\wh{A})_\Q$ maps $G_i(A)_\Q$ to $G_i(\wh{A})_\Q$  injectively, it follows that $[U] \in G_1(A)_\Q$.

Case-2: $Q = P_j$ for some $j$, say $Q = P_1$. \\
By prime avoidance we can choose a regular sequence $\xb = x_1, \ldots, x_{d-1} \in P_1 \setminus \bigcup_{j = 2}^{r}P_j$. Note $[\wh{A}/(\xb)] = 0$ in $G(\wh{A})$. We choose a filtration
$0 = H_0 \subsetneq H_1 \subsetneq\cdots  \subsetneq H_{s-1}   \subsetneq H_s = \wh{A}/(\xb)$ with $H_i/H_{i-1} = \wh{A}/Q_i$ for some prime $Q_i$. We note that $\height Q_i \geq d - 1$ and $Q_i \neq P_j$ for every $i \geq 1$ and $j \geq 2$. As $\Ass \wh{A}/(\xb) \subseteq \{ Q_1, \ldots, Q_s \}$ it follows that $Q_i = P_1$ for some $i$.
Thus we have an equation in $G(\wh{A})$
\[
0 = [\wh{A}/(\xb)] = a [\wh{A}/P_1] + b [\wh{A}/\wh{\m}] + \sum_{i = 1}^{s-2}b_i[\wh{A}/Q_i], \quad \text{where} \ a \geq 1, b \geq 0, \ \text{and}\ b_i \geq 0,
\]
where $Q_i \neq P_j$ for all $i$ and $j$. We note that $t[\wh{A}/\wh{\m}] = 0$ for some $t \geq 1$. By Case-1 it follows that there exists $[W] \in G_1(A)_\Q$ which maps to $[\wh{A}/P_1]$ in $G(\wh{A})_\Q$.

Thus the map $G_1(A)_\Q \rt G_1(\wh{A})_Q$ is bijective. The result follows.
\end{proof}
We now give
\begin{proof}[Proof of \ref{chow-2}]
It suffices to show that the natural map $A_i(A)_\Q \rt A_i(\wh{A})_\Q$ is an isomorphism for all $i = 0, 1, 2$. We note that it is injective. As $\wh{A}$ (and so $A$) are domains we have
$A_2(A)_\Q = A_2(\wh{A})_\Q = \Q$. By \ref{chow-1} the map $A_1(A)_\Q \rt A_1(\wh{A})_\Q$ is an isomorphism. We note that $A_0(A)_\Q = A_0(\wh{A})_\Q = 0$. The result follows.
\end{proof}
Next we give
\begin{proof}[Proof of \ref{chow-3}]
We note that $A_2(A) = C(A)$. Thus the natural map $A_2(A)_\Q \rt A_2(\wh{A})_\Q$ is an isomorphism if and only if the natural map $C(A)_\Q \rt C(\wh{A})_\Q$ is an isomorphism, see \cite[2.3]{KK}.
We consider the natural map $A_i(A)_\Q \rt A_i(\wh{A})_\Q$ for all $i = 0, 1, 2, 3$. We note that it is injective. As $\wh{A}$ (and so $A$) are domains we have
$A_3(A)_\Q = A_3(\wh{A})_\Q = \Q$. By \ref{chow-1} the map $A_1(A)_\Q \rt A_1(\wh{A})_\Q$ is an isomorphism. We note that $A_0(A)_\Q = A_0(\wh{A})_\Q = 0$. Thus the natural map
$A_*(A)_\Q \rt A_*(\wh{A})_\Q$ is an isomorphism if and only if  the natural map $C(A)_\Q \rt C(\wh{A})_\Q$ is an isomorphism. The result follows.
\end{proof}
\section{Examples}
We give several examples where our results hold.

\s Let $(A,\m)$ be an excellent Gorenstein equi-characteristic ring of positive even dimension such that $A$ has finite representation type. Assume the residue field $k$ of $A$ is perfect. Then $G(A)_\Q = \Q$ and $G(\wh{A})_\Q = \Q$, see \cite[1.5]{P-O}. As the natural map $\eta \colon G(A) \rt G(\wh{A})$ is injective it follows $\eta_\Q $ is an isomorphism

\s Let $S = \bigoplus_{n \geq 0}S_n$ be a graded (not necessarily standard) two dimensional Gorenstein normal domain over a field $k = S_0$. Assume that $k$ is a finite field or is the  algebraic closure of a finite field.
Let $A = S_{S_+}$ where $S_+ = \bigoplus_{n \geq 1}S_n$. As $A$ is excellent we have that $\wh{A}$ is a two dimensional normal domain. By the non-trivial results in
\cite[Theorem 4]{CS} and \cite[4.5]{G}
it follows that $C(\wh{A})_\Q = 0$ (here $C(\wh{A})$ is the class group of $\wh{A}$).
It is well known that if $R$ is a two dimensional Gorenstein normal domain then $G(\CMS(R))_\Q = C(R)_\Q$; see \cite[2.5]{CD}.
Thus $G(\CMS(\wh{A}))_\Q = 0$. As $\eta \colon G(\CMS(A))_\Q \rt G(\CMS(\wh{A})_\Q$ is injective, it follows that $G(\CMS(A))_\Q = 0$. So $G(A)_\Q \cong G(\wh{A})_\Q$, see \ref{main3-body}.

\s Let $R = \mathbb{C}[X_1, \ldots, X_n]$ with $n \geq 2$ and let $G \subseteq SL_n(\Cb)$ be a finite group acting linearly on $R$. Let $S = R^G$ be the ring of invariants. Then $S$ is Gorenstein, see \cite[section 4, Theorem 1]{WK}. Assume $S$ is an isolated singularity. Let $A = S_{S_+}$. Note $\wh{A} = \Cb[[X_1, \ldots, X_n]]^G$. Then by \cite[Chapter 3, 5.6]{AR} it follows that $G(\wh{A}) = \Z \oplus H$ where $H$ is a finite abelian group.
We note that we have a linear surjective map $G(A) \rt \Z$ given by the rank function. It follows that $\dim_\Q G(A)_\Q \geq 1$. As the natural map $\eta_\Q \colon G(A)_\Q \rt G(\wh{A})_\Q$ is injective and as $G(\wh{A})_\Q  = \Q$ it follows that $\eta_\Q$ is an isomorphism.

\s  Let $R = k[t^{a_1}, \cdots, t^{a_m}]$ be a symmetric numerical semi-group ring. Then $R$ is Gorenstein, see \cite[4.4.8]{BH}.  Set $S = R[X]$. Let
$$ B = S_{(t^{a_1}, \cdots, t^{a_m}, X)}.$$ The completion of $B$ is
$k[[t^{a_1}, \cdots, t^{a_m}]][[X]]$  which is a domain.  Let $A$ be the Henselization of $B$. Then by \ref{chow-2} we have $\eta_\Q \colon G(A)_\Q \rt G(\wh{A})_\Q$ is an isomorphism.

\s Let $R = k[X_1, X_2, X_3]$  and let $G \subseteq SL_3(k)$ be a finite group acting linearly on $R$. Assume order of $G$ is invertible in $k$. Let $S = R^G$ be the ring of invariants. Then $S$ is Gorenstein, see \cite[section 4, Theorem 1]{WK}. Let $A = S_{S_+}$. Note $\wh{A} = k[[X_1, X_2, X_3]]^G$. Then by \cite[Chapter 3, 7.2]{AR} it follows that $C(\wh{A})$, the class group of $\wh{A}$ is a finite abelian group. Set $B$ to be the Henselization of $A$. So $\wh{B} = \wh{A}$. Then as the map $C(B) \rt C(\wh{B})$ is injective it follows that the $C(B)$ is also a finite group. Thus $C(B)_\Q = C(\wh{B})_Q = 0$. We note that $B$ is normal  and so it is $R_1$. By \ref{chow-2} we have $\eta_\Q \colon G(B)_\Q \rt G(\wh{B})_\Q$ is an isomorphism.

%\section*{Acknowledgements}

\end{document}